\documentclass[11pt,reqno]{amsart}
\usepackage{amssymb}
\usepackage{amsfonts}
\usepackage{amsmath}
\usepackage{stmaryrd}
\usepackage{physics}
\usepackage{braket}
\usepackage{dsfont}
\usepackage{bbold}
\usepackage{graphicx}
\usepackage{relsize}
\usepackage{makecell}
\usepackage[table,xcdraw]{xcolor}
\usepackage{enumerate}
\usepackage[pagebackref, colorlinks = true, linkcolor = blue, urlcolor  = blue, citecolor = red]{hyperref}
\usepackage[margin=1in]{geometry}
\usepackage{enumitem}
\usepackage{mathtools}
\usepackage{graphbox}
\usepackage{comment}
\usepackage{float}
\usepackage[capitalize]{cleveref}
\usepackage{booktabs}
\usepackage[skins]{tcolorbox}
\usepackage{nicematrix}
\NiceMatrixOptions{cell-space-limits = 2pt}

\renewcommand{\epsilon}{\varepsilon}
\renewcommand{\phi}{\varphi}

\newcommand{\M}[1]{\mathcal{M}_{#1}(\mathbb{C})}

\newcommand{\Mreal}[1]{\mathcal{M}_{#1}(\mathbb{R})}

\newcommand{\U}[1]{\mathcal{U}(#1)}

\DeclareMathOperator{\diag}{diag}
\DeclareMathOperator{\rk}{rk}

\newtheorem{theorem}{Theorem}[section]
\newtheorem{definition}[theorem]{Definition}
\newtheorem*{definition*}{Definition}
\newtheorem{proposition}[theorem]{Proposition}
\newtheorem{corollary}[theorem]{Corollary}
\newtheorem{lemma}[theorem]{Lemma}
\newtheorem{remark}[theorem]{Remark}

\newtheorem*{conjecture*}{Conjecture}

\newcommand\vertarrowbox[3][6ex]{
    \begin{array}[t]{@{}c@{}} #2 \\
        \left\uparrow\vcenter{\hrule height #1}\right.\kern-\nulldelimiterspace\\
        \makebox[0pt]{\scriptsize#3}
    \end{array}
}

\theoremstyle{definition}
\newtheorem{example}[theorem]{Example}

\author{Ion Nechita}
\email{ion.nechita@univ-tlse3.fr}
\address{Laboratoire de Physique Th\'eorique, Universit\'e de Toulouse, CNRS, UPS, France}

\author{Zikun Ouyang}
\email{zikun.ouyang@gmail.com}
\address{Institut de Math\'ematiques, Universit\'e de Toulouse, CNRS, UPS, France}

\author{Anna Szczepanek}
\email{anna.szczepanek@math.univ-toulouse.fr}
\address{Institut de Math\'ematiques, Universit\'e de Toulouse, CNRS, UPS, France}

\title{Generalized unistochastic matrices}

\begin{document}
    \begin{abstract}
        We study a class of bistochastic matrices generalizing unistochastic matrices. Given a complex bipartite unitary operator, we construct a bistochastic matrix having as entries the normalized squared Frobenius norm of the blocks. We show that the closure of the set of generalized unistochastic matrices is the whole Birkhoff polytope. We characterize the points on the edges of the Birkhoff polytope that belong to a given level of our family of sets, proving that the different (non-convex) levels have a rich inclusion structure. We also study the corresponding generalization of orthostochastic matrices. Finally, we introduce and study the natural probability measures induced on our sets by the Haar measure of the unitary group. These probability measures interpolate between the natural measure on the set of unistochastic matrices and the Dirac measure supported on the van der Waerden matrix.
    \end{abstract}

    \maketitle

    \tableofcontents

    \section{Introduction}

    Bistochastic and unistochastic matrices is a classical topic that comes up repeatedly in various domains of mathematics and mathematical physics. To set the stage and introduce notation, let us recall that a square real matrix of size $d$, the set of which we shall denote by $\Mreal{d}$, is called \emph{bistochastic} (or \emph{doubly stochastic}) if it has non-negative entries that add up to one in every row and column. That is,
    $ B = (B_{ij})_{i,j=1}^d$ is bistochastic when its entries satisfy the following conditions:
    $$\forall {i,j}\quad B_{ij}\geq 0\qquad \forall {i}\quad \sum_{j = 1}^dB_{ij}=1 \quad \text{ and }\quad \forall {j }\quad \sum_{i = 1}^dB_{ij}=1.
    $$
    We shall denote by $\mathsf B_d$ the set of all bistochastic matrices of order $d$.

    A bistochastic matrix is called \emph{unistochastic} when its entries are the squared absolute values of some unitary matrix of the same size. More formally,
    we consider the map
    $$
    \begin{aligned}
        \phi_d\colon  \U{d} \ni \left(U_{ij}\right)_{i,j=1}^d \longmapsto \left(|U_{ij}|^2\right)_{i,j=1}^d  \in \Mreal{d}
    \end{aligned}.
    $$
    The image of $\U{d}$ under $\phi_d$ constitutes the set $\mathsf U_d$ of unistochastic matrices of order $d$. Alternatively, using the Hadamard (entrywise) product $\circ$ of matrices, we can write
    $$\mathsf U_d:= \phi_d( \U{d}) = \{ U\circ\bar U \:|\:U\in \U{d} \}.$$
    By unitarity, we have $\mathsf U_d\subseteq \mathsf B_d$.

    One of the reasons behind the prominence of bistochastic matrices is the fact that its entries can be regarded as the  probabilities that some (classical) physical system evolves from one state to another. If the bistochastic  matrix is also unistochastic, then the system under consideration can be quantized. There are many references related to the applications of unistochastic matrices in various areas, e.g., in quantum information theory and in particle physics, see \cite{bengtsson2004import} and references therein.

    \smallskip

    It is well known that the \emph{Birkhoff polytope} $\mathsf B_d$ is convex and compact. The extreme points of $\mathsf B_d$ are permutation matrices, so the Birkhoff polytope  has $d!$ vertices. A bistochastic matrix lies at the boundary of the Birkhoff polytope iff it has a zero entry. There are $d^2$  faces and they correspond to the inequalities that the matrix entries must satisfy. For instance,
    consider $d=3$ and
    $$B=\begin{bmatrix}
        a & b & *\\
        c & d & *\\
        * & * & *
    \end{bmatrix}$$
    with $a,b,c,d \in \mathbb R$. The $9$ inequalities corresponding to the faces of the Birkhoff polytope are
    $$
    \begin{cases}
        a,b,c,d\geq 0   \\
        a+b\leq 1,\quad c+d\leq 1,\quad a+c\leq 1,\quad b+d\leq 1  \\
        (1-a-c)+(1-b-d)\leq 1
    \end{cases}
    $$

    An obvious example of a unistochastic matrix is a permutation matrix. Another well-known example is the van der Waerden (flat)	matrix, i.e., the matrix whose entries are all equal to $1/d$. It is worth noting that the unitary matrices that induce the van der Waerden
    matrix are precisely the renowned complex Hadamard matrices, for instance the Fourier matrix $\frac{1}{\sqrt d}(\exp(\frac{2\pi\mathrm{i}}{d}jk))_{j,k=0}^{d-1}$.

    One immediately sees that for $d=2$ every bistochastic matrix is unistochastic, i.e., we have $\mathsf U_2 = 	\mathsf B_2$. This, however, is a sole exception as for every dimension $d$ higher than two  we have $\mathsf U_d \subsetneq   \mathsf B_d$ and $\mathsf U_d$ is known to be non-convex. A lot of effort has been put into characterizing unistochastic matrices and one of the key tools turned out to be the \emph{bracelet condition}. It allows us to distinguish the set of \emph{bracelet matrices}, which is a superset of unistochastic matrices.
    Namely, let $\alpha = (\alpha_1, \ldots, \alpha_d)$ and $\beta = (\beta_1, \ldots, \beta_d)$ be probability vectors (i.e., in each vector the entries are non-negative and sum up to one). We say that $(\alpha,\beta)$ satisfies the bracelet condition if
    \begin{equation}\label{eq:intro-Bd}
        2\max_{j=1, \ldots, d}\sqrt{\alpha_{j}\beta_{j}}\leq\sum_{j=1}^d\sqrt{\alpha_{j}\beta_{j}}.
    \end{equation}
    Now, a bistochastic matrix is said to be a \emph{bracelet matrix} if every pair of its rows and every pair of its columns, regarded as pairs of probability vectors, satisfies the bracelet condition; we shall denote the set of bracelet matrices of order $d$ by $\mathsf L_d$.
    The bracelet condition plays an instrumental role in the study of unistochastic matrices because it characterizes the first non-trivial case $d=3$ \cite{au1979orthostochastic}, i.e.,
    $$\mathsf U_3 = \mathsf L_3 \subsetneq \mathsf B_3$$
    and, as we already mentioned, it provides a necessary (but not sufficient) condition for unistochasticity in higher dimensions (see \cite{rajchel2022algebraic}):
    $$\forall d \geq 4, \qquad \mathsf U_d \subsetneq \mathsf L_d \subsetneq \mathsf B_d.$$

    In the  present paper we introduce the notion of \emph{generalized unistochastic matrices}, denoted by $\mathsf U_{d,s}$;
    here, $s$ is an integer parameter.
    The idea is to replace $\U{d}$ with $\U{ds}$ and regard a unitary matrix from $\U{ds}$ as a $d \times d$ matrix consisting of $s \times s$ submatrices (blocks). Then it suffices to replace the absolute values of entries by the  normalized squares of Frobenius (or Schatten-2) norms of blocks to arrive at a $d \times d$ bistochastic matrix again. Formally, we consider the map
    $$
    \phi_{d,s}\colon  \U{d s} \ni  \big( U_{ij}(k,l)\big)_{\substack{ 1\leq i,j\leq d \\ 1\leq k,l \leq s}} \  \longmapsto   \big(\tfrac{1}{s}||U_{ij}||_F^2\big)_{1\leq i,j\leq d} \in \mathsf B_d,
    $$
    where   $U_{ij}(k,l)$ is the $(k,l)$-th entry in the $(i,j)$-th block, and we define $\mathsf U_{d,s}$ as the image of $\phi_{d,s}$. Let us recall  that the Frobenius norm of a square complex matrix $X \in \M{n}$ is given by
    $$\|X\|_F := \Big( \sum_{i,j=1}^n |X_{ij}|^2 \Big)^{1/2} = \Tr(XX^*)^{1/2}.$$
    In Propositions \ref{prop:Uds_in_Bd} \& \ref{prop:Ud_in_Uds} we show that generalized unistochastic matrices do indeed generalize  the notion of unistochastic matrices, i.e., for every $s$ we have
    $$\mathsf U_d \subseteq \mathsf U_{d,s} \subseteq \mathsf B_d.$$

    One of the main results of the present paper is \cref{thm:union-is-all-bistochastic}, where we show that every bistochastic matrix can be arbitrarily well approximated by a generalized unistochastic matrix of some order. The key ingredient in  proving this result is the convexity-type property of generalized unistochastic matrices:
    $$\forall {s,t} \quad \tfrac{s}{s+t}\mathsf U_{d,s}+\tfrac{t}{s+t}\mathsf U_{d,t}\subseteq \mathsf U_{d,s+t},$$
    see \cref{prop:s-convexity-Uds}. Then in \cref{Corollary} we  investigate further non-trivial inclusion relations between the sets of generalized unistochastic matrices of different orders.

    Let us point our that an alternative generalization of unistochastic matrices was proposed by Gutkin in \cite{GUTKIN201328}. Unfortunately, as we show at the end of \cref{sec:generalized-unistochastic-matrices},
    the proposed generalization yields only stochastic (and generally not bistochastic) matrices, so it is quite far from the usual unistochastic matrices.

    Generalized unistochastic matrices were also considered in \cite{shahbeigi2021log}, as classical channels associated to generalized unistochastic channels. More precisely, given a bipartite unitary matrix $U \in \mathcal U(ds)$, Shahbeigi, Amaro-Alcal{\'a}, Pucha{\l}a, and {\.Z}yczkowski consider the quantum channel $\Phi : \M{d} \to \M{d}$ given by
    $$\Phi(X) = \Tr_s \left[ U \left(X \otimes \frac{I_s}{s}\right) U^* \right].$$
    The classical transition matrix corresponding to this channel, $B_{ij} := \langle i | \Phi(\ketbra{j}{j}) | i \rangle$ corresponds precisely to the generalized bistochastic matrices we study. In this work, we further the understanding of these objects, providing new insights on their structure and relation to uni- and bi-stochastic matrices. We shall refer to \cite{shahbeigi2021log} at different points of this paper, emphasizing the new contributions of our research. Our focus will be on generalized unistochastic matrices, and not on the unistochastic channels, as in \cite{shahbeigi2021log}.

    \smallskip

    The main tool we develop to investigate $\mathsf U_{d,s}$ is the \emph{generalized bracelet condition}. A pair of probability vectors $\alpha,\beta$ is said to satisfy the generalized bracelet condition of order $s$ if they correspond to the normalized squares of Frobenius norms of the blocks of some unitary matrix $U \in \mathcal {U}(d s)$, i.e.,
    if there exist  matrices $A_1, \ldots, A_d, B_1, \ldots, B_d \in \M{s}$ satisfying
    $$ \forall i  \quad \tfrac{1}{s} \|A_i\|^2_F =\alpha_i\quad \text{ and } \quad \tfrac{1}{s} \|B_i\|^2_F=\beta_i$$
    as well as
    $$\sum_{i=1}^d A_iA_i^*=\sum_{i=1}^d B_iB_i^*=I_s, \quad \sum_{i=1}^d A_iB_i^*=0,$$
    which means that the $2$-row block matrix  $$
    \begin{bmatrix}
        A_1 & \cdots & A_d \\
        B_1 & \cdots & B_d
    \end{bmatrix}$$
    (of  size  $2s \times ds$) can be expanded to a unitary matrix of size $ds \times ds$.  See Proposition \ref{prop:Bd1_is_Bd} for the proof  that these conditions do indeed generalize the standard bracelet condition  \eqref{eq:intro-Bd}.

    \smallskip

    Since the generalized unistochastic matrices $\mathsf U_{d,s}$ are defined as the image of $\mathcal{U}(ds)$ under $\phi_{d,s}$, it is natural to equip $\mathsf U_{d,s}$ with the probability measure obtained by pushing forward the Haar measure from $\mathcal{U}(ds)$ via  $\phi_{d,s}$. We compute the first few joint moments of the elements of a random matrix $B \in \mathsf U_{d,s}$. In particular, the expected value of $B_{ij}$ equals $\ 1/d$, while its variance decreases as $s$ grows (and $d$ is fixed), which means that the probability distribution on $\mathsf U_{d,s}$ tends to concentrate around the van der Waerden matrix. We also draw some conclusions regarding the covariance and correlation of the elements of $B$.

    \bigskip

    The paper is organized as follows. In \cref{sec:generalized-unistochastic-matrices} we introduce generalized unistochastic matrices and present their basic properties. \cref{sec:bracelet} contains a suitable generalization of the bracelet conditions for unistochastic matrices; these conditions allow us   to showcase the complexity of the different levels of the generalized unistochastic sets. In \cref{sec:orthostochastic} we discuss the corresponding generalizations of orthostochastic matrices. Finally, in \cref{sec:random} we explore the properties of the probability measures induced on the set of generalized unistochastic matrices by the Haar distribution on the unitary group.

    \section{Generalized unistochastic matrices} \label{sec:generalized-unistochastic-matrices}

    The main idea of this work can be summarized in the following table:

    \medskip

    \begin{center}
        \begin{tcolorbox}[enhanced,width=6.3in,center upper,drop fuzzy shadow southwest,
            colframe=red!50!black,colback=yellow!5]
            \bgroup
            \def\arraystretch{2}
            \begin{tabular}{@{}c|c|c@{}}
                \textit{Source} & \textit{Operation} & \textit{Result} \\ \hline\hline
                \begin{tabular}[c]{@{}c@{}}Unitary group\\[-0.5em] $U \in \mathcal U(d)$\end{tabular} & $\mathbb C \ni u_{ij} \mapsto |u_{ij}|^2$ & \begin{tabular}[c]{@{}c@{}}Unistochastic matrix\\[-0.5em] $B \in \mathsf U_d$\end{tabular} \\ \hline
                \begin{tabular}[c]{@{}c@{}}Larger unitary group\\[-0.5em] $U \in \mathcal U(ds) \subseteq \mathcal M_d(\mathcal M_s(\mathbb C))$\end{tabular} & $
                \mathcal M_s(\mathbb C)\ni U_{ij} \mapsto \frac{1}{s}||U_{ij}||_F^2$ & \begin{tabular}[c]{@{}c@{}}Generalized unistochastic matrix\\[-0.5em] $B \in \mathsf U_{d,s}$\end{tabular}
            \end{tabular}
            \egroup
        \end{tcolorbox}
    \end{center}

    \medskip

    Let $d \geq 2$ and $s \geq 1$ be integers. In what follows we regard $B \in \M{d s}$ as a $d \times d$ block matrix consisting of $s \times s$ blocks. We shall write $B_{ij}$ for the the $(i,j)$-th block and $B_{ij}(k,l)$ for the \mbox{$(k,l)$-th} coefficient inside this block, where $k,l \in [s]$ and $i,j \in [d]$. For brevity, we put $[n]:=\{1,2, \ldots, n\}$ and $\mathfrak S_n$  for the group of permutations of $[n]$. We come now to the main definition of this work, that of generalized unistochastic matrices. These objects have previously been considered in \cite{shahbeigi2021log}, in relation to classical actions of quantum channels.

    \begin{definition}\label{def:generalized-unistochastic}
        Consider the map
        $$
        \begin{aligned}
            \label{eq:def-phi-d-s}
            \phi_{d,s}\colon &\U{d s}&\xlongrightarrow{}&\  \Mreal{d}\\
            &\left( U_{ij}(k,l)\right)_{i,j \in [d];  k,l \in [s]}     &\longmapsto     &\:\big(\tfrac{1}{s}||U_{ij}||_F^2\big)_{ i,j\in [d]}
        \end{aligned}
        $$
        We define $\mathsf U_{d,s}:=\phi_{d,s}(\U{d s})$ to be the set of \emph{generalized unistochastic matrices}. A matrix $B$ in the range of $\phi_{d,s}$ will be called $s$-unistochastic \cite{shahbeigi2021log}.
    \end{definition}

    \begin{example}
        For $d=s=2$, $P=\begin{bNiceMatrix}[c,margin]
            \CodeBefore
            \rectanglecolor{blue!10}{1-1}{2-2}
            \rectanglecolor{red!10}{3-1}{4-2}
            \rectanglecolor{red!10}{1-3}{2-4}
            \rectanglecolor{blue!10}{3-3}{4-4}
            \Body
            0 & 0 & 0 & 1\\
            0 & 0 & 1 & 0\\
            0 & 1 & 0 & 0\\
            1 & 0 & 0 & 0
        \end{bNiceMatrix}  \stackrel{\phi_{2,2}}{ \longmapsto }
        \begin{bNiceMatrix}[c,margin]
            \CodeBefore
            \chessboardcolors{blue!10}{red!10}
            \Body
            0 & 1\\
            1 & 0
        \end{bNiceMatrix} \in \mathsf B_2.
        $

    \end{example}

    Importantly, there exist generalized unistochastic matrices which are not unistochastic. This makes the definition above interesting and justifies the study of generalized unistochastic matrices.

    \begin{example}\label{ex:generalized-but-not-unistochastic}
        For $d=3$, $ P=\begin{bNiceMatrix}[c,margin]
            \CodeBefore
            \rectanglecolor{blue!10}{1-1}{2-2}
            \rectanglecolor{red!10}{1-3}{2-4}
            \rectanglecolor{blue!10}{1-5}{2-6}
            \rectanglecolor{red!10}{3-1}{4-2}
            \rectanglecolor{blue!10}{3-3}{4-4}
            \rectanglecolor{red!10}{3-5}{4-6}
            \rectanglecolor{blue!10}{5-1}{6-2}
            \rectanglecolor{red!10}{5-3}{6-4}
            \rectanglecolor{blue!10}{5-5}{6-6}
            \Body
            0 & 0 & 0 & 0 & 1 & 0\\
            0 & 0 & 0 & 1 & 0 & 0\\
            1 & 0 & 0 & 0 & 0 & 0\\
            0 & 0 & 0 & 0 & 0 & 1\\
            0 & 1 & 0 & 0 & 0 & 0\\
            0 & 0 & 1 & 0 & 0 & 0
        \end{bNiceMatrix}  \stackrel{\phi_{3,2}}{ \longmapsto } B=
        \begin{bNiceMatrix}[c,margin]
            \CodeBefore
            \chessboardcolors{blue!10}{red!10}
            \Body
            0 & \frac{1}{2} & \frac{1}{2}\\[0.33em]
            \frac{1}{2} & 0 &\frac{1}{2}\\[0.33em]
            \frac{1}{2} & \frac{1}{2} & 0
        \end{bNiceMatrix} \in \mathsf U_{3,2} \setminus \mathsf U_{3}$.

        \medskip

        \noindent Indeed, $P$ is a permutation matrix corresponding to the permutation $(1\, 3\, 6\, 4\, 2\, 5) \in \mathfrak S_6$, hence $B$ is 2-unistochastic, i.e.,  $B \in \mathsf U_{3,2}$. However, $B$ is not unistochastic, i.e., $B \notin \mathsf U_{3}$,  since it does not satisfy the bracelet conditions, see  \cref{sec:bracelet}.
    \end{example}

    In the next two propositions, we show that generalized unistochastic matrices are bistochastic and that they contain, for every value of the parameter $s$, the set of (usual) unistochastic matrices $\mathsf U_d$, which coincides with the generalized family at $s=1$. The special case $s=d$ of the latter result, relevant in the study of some class of quantum channels, has been considered in \cite[Proposition 21]{shahbeigi2021log}.

    \begin{proposition}\label{prop:Uds_in_Bd} For every    $d \geq 2$ and $s \geq 1$, we have
        $\mathsf U_{d,s}\subseteq \mathsf B_d$.
    \end{proposition}
    \begin{proof}
        Let $d \geq 2$ and $s \geq 1$.
        By unitarity,
        $$\sum_{i=1}^d\sum_{k=1}^s |U_{ij}(k,l)|^2=\sum_{j=1}^d\sum_{l=1}^s |U_{ij}(k,l)|^2 = 1.
        $$
        Therefore, for all $i \in [d]$ we have
        $$ \sum_{j=1}^d B_{ij} = \tfrac{1}{s}\sum_{j=1}^d \|U_{ij}\|_F^2  =\tfrac{1}{s} \sum_{k=1}^s\sum_{j=1}^d\sum_{l=1}^s|U_{ij}(k,l)|^2=1$$
        and, analogously, $ \sum_{i=1}^d B_{ij} = 1$, which concludes the proof.
    \end{proof}

    \begin{proposition}\label{prop:Ud_in_Uds}
        For every $d \geq 2$ and $s \geq 1$, we have   $\mathsf U_d=\mathsf U_{d,1}\subseteq \mathsf U_{d,s}$.
    \end{proposition}
    \begin{proof}Let $d \geq 2$ and $s \geq 1$, and
        let $B \in \mathsf U_d$, i.e., there exists   $U\in \U{d}$ such that $B = \phi_d(U)  $. Consider $V:=U\otimes I_s\in \U{d s}$. Then for all $i, j \in [d]$ we have   $V_{ij} = U_{ij} \otimes I_s$, which implies that   $$\tfrac{1}{s}\|V_{ij}\|_F^2=|U_{ij}|^2 = B_{ij};$$ hence, $B \in \mathsf U_{d,s}$, as desired.
    \end{proof}

    Next, we show that the sets $\mathsf U_{d,s}$ satisfy a kind of convexity property. This result will be key in showing one of our main results, Theorem \ref{thm:union-is-all-bistochastic}.

    \begin{proposition}\label{prop:s-convexity-Uds} For every $d \geq 2$  and $s,t \geq 1$, we have
        $\frac{s}{s+t}\mathsf U_{d,s}+\frac{t}{s+t}\mathsf U_{d,t}\subseteq \mathsf U_{d,s+t}$.
    \end{proposition}
    \begin{proof}
        Fix $d \geq 2$ and $s,t \geq 1$, and let $B\in \mathsf U_{d,s}$ and $C\in \mathsf U_{d,t}$. There exist $V\in \U{d   s}$ and $W\in \U{d  t}$ such that  $B_{ij}= \frac{1}{s} \|V_{ij}\|_F^2$ and $C_{ij}= \frac{1}{t} \|W_{ij}\|_F^2$. 	Consider $U\in \M{d  (s+t)}$ defined as
        $$U_{ij}:=
        \begin{bmatrix}
            V_{ij} & 0 \\
            0 & W_{ij}
        \end{bmatrix}.
        $$
        In particular, up to a permutation of blocks, $U$ coincides with $V\oplus W$. Thus,  $U \in  \U{d  (s+t)}$ and
        $$\tfrac{1}{s+t} \|U_{ij}\|_F^2=
        \tfrac{1}{s+t} (\|V_{ij}\|_F^2+\|W_{ij}\|_F^2)=
        \tfrac{1}{s+t} (sB_{ij}+tC_{ij}).$$
        That is, $
        \phi_{d,t+s}(U)  = \frac{s}{s+t}B+\frac{t}{s+t}C   \in  \mathsf U_{d,s+t}$, as desired.
    \end{proof}

    \begin{corollary}
        \label{Corollary}Let $d \geq 2$. From Proposition \ref{prop:s-convexity-Uds} we easily conclude that
        \begin{enumerate}
            \item
            For all orders $s_1, \ldots, s_k \geq 1$,  we have
            $$\frac{s_1}{s_1+\ldots+s_k}\mathsf U_{d,s_1}+\ldots+
            \frac{s_k}{s_1+\ldots+s_k}\mathsf U_{d,s_k}
            \subseteq \mathsf U_{d,s_1+\ldots+s_k}.$$

            \item For all $s, n \geq 1$,  we have $\mathsf U_{d,s} \subseteq \mathsf U_{d,ns}$.

            \item  For all $s,t \geq 1$,  we have    $\mathsf U_{d,s}\cap \mathsf U_{d,t} \subseteq \mathsf U_{d,s+t}$.
        \end{enumerate}
    \end{corollary}

    In relation to the second point of the corollary above, note that, in general, we do not have
    $$s \leq t \implies \mathsf U_{d,s} \subseteq \mathsf U_{d,t},$$
    see Corollary \ref{cor:not-increasing-in-s} for counterexamples in this direction.

    We now prove the main theorem of this section: the closed union of all generalized unistochastic matrices constitutes the whole set of bistochastic matrices.

    \begin{theorem}\label{thm:union-is-all-bistochastic}

        For every dimension $d \geq 2$, we have
        $$\overline{\bigcup_{s\geq 1}\mathsf U_{d,s}}=\mathsf B_d.$$
    \end{theorem}

    \begin{proof}
        We only need to prove the ``$\supseteq$'' inclusion. Fix $d \geq 2$ and $\epsilon>0$.
        Let $B\in \mathsf B_d$. We shall construct $N \in \mathbb N$ and $B_\epsilon \in \mathsf U_{d,N}$ such that $\|B-B_\epsilon\|_F\leq \epsilon$.
        As a bistochastic matrix, $B$ can be written as a convex combination of permutation matrices, i.e., there exists a family $\{t_\sigma\, |\, \sigma \in \mathfrak S_d\}$ of non-negative coefficients such that  $\sum_{\sigma \in \mathfrak S_d} t_\sigma=1$ and $B=\sum_{\sigma \in \mathfrak S_d} t_\sigma P_\sigma$, where $P_\sigma$ is the permutation matrix corresponding to $\sigma \in \mathfrak S_d$.
        Take $\delta :=   {\epsilon}/[{2(d!-1)\sqrt{d}}]  $ and let  $N$ be large enough so that $$\max_{\sigma\neq id}(t_\sigma-\tfrac{k_\sigma}{N})\leq \delta,$$
        where $k_\sigma:=\lfloor Nt_\sigma \rfloor$. Define $k_{id}:=N -\sum_{\sigma\neq id }\ {k_\sigma}$.
        Then $$\tfrac {k_{id} }N - t_{id} =  \sum_{\sigma\neq id }(t_\sigma - \tfrac{k_\sigma}{N}) \in [0, \delta (d! - 1)].$$

        Let us now consider $B_\epsilon:=\sum_{\sigma \in \mathfrak S_d} \frac{k_{\sigma}}N P_\sigma$. Since $P_\sigma \in \mathsf U_d \subset \mathsf U_{d, k_\sigma}$ if $k_\sigma \neq 0$, from Corollary~\ref{Corollary} it follows that  $B_\epsilon \in \mathsf U_{d,N}$. Therefore,
        $$ \|B-B_\epsilon\|_F= \Big\|\sum_{\sigma \in \mathfrak S_d} (t_\sigma-\tfrac{k_\sigma}{N})P_\sigma\Big\|_F\leq \sum_{\sigma \in \mathfrak S_d}|t_\sigma-\tfrac{k_\sigma}{N}|\cdot \sqrt d\leq 2(d!-1)\delta\sqrt{d}=\epsilon,$$
        as desired.
    \end{proof}

    \bigskip

    Next, we present  some numerical simulations regarding the set $\mathsf U_{3,1}$ of $3 \times 3$ unistochastic matrices and its generalized version $\mathsf U_{3,2}$, see Figure \ref{fig:random-U3s}. To decide whether a bistochastic matrix $B$ is an element of $\mathsf U_{d,s}$, we use the $\texttt{NMinimize}$ function of Wolfram Mathematica to try finding a unitary matrix $U \in \U{ds}$ such that $\phi_{d,s}(U) = B$. This method does not guarantee finding the global minimum of non-convex functions, so the results in Figure \ref{fig:random-U3s} are empirical; note however that there is a perfect fit with the theory in the case $(d,s) = (3,1)$.

    \begin{figure}[htb]
        \centering
        \includegraphics[width=0.475\textwidth, height=0.4\textwidth]{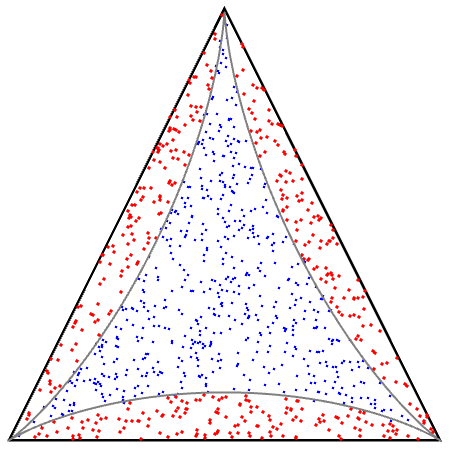} \quad
        \includegraphics[width=0.475\textwidth, height=0.4\textwidth]{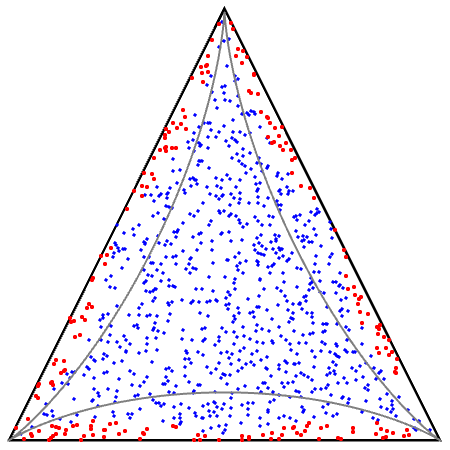}
        \caption{In the simplex defined by the identity, $(123)$, $(321)$ permutation matrices, we plot 1000 uniformly sampled random bistochastic matrices. In the left panel, we plot in blue unistochastic elements, i.e.,~samples $B \in \mathsf U_{3,1}$, and in red samples outside this set. In the right panel, we use the same colours to plot the elements inside and outside the set $\mathsf U_{3,2}$. The gray curves correspond to the bracelet conditions \eqref{eq:bracelet-3} characterizing unistochastic $3 \times 3$ matrices, see Section \ref{sec:bracelet}.}
        \label{fig:random-U3s}
    \end{figure}

    \bigskip

    We end this section by discussing Gutkin's generalization of unistochastic matrices \cite{GUTKIN201328}. In that paper, the author generalizes unistochastic matrices starting from an isometry $$V\colon \mathbb C^d \to \mathbb C^d \otimes \mathbb C^s.$$ This isometry can be seen as a $\mathbb C^s$-valued $d \times d$ matrix:
    $$\mathcal M_{ds \times d}(\mathbb C) \cong \mathcal M_d(\mathbb C^s).$$
    Denoting the vector elements of $V$ as $v_{ij} \in \mathbb C^s$, where $i,j \in [d]$, Gutkin defines
    $$B_{ij} := \|v_{ij}\|^2.$$
    Unfortunately, in general the resulting matrix $B$ is then only column stochastic, and not row stochastic. This can be seen, for example, by considering the case $d=s=2$ and the isometry
    $$V = \begin{bmatrix}
        1 & 0 \\
        0 & 1 \\
        0 & 0 \\
        0 & 0
    \end{bmatrix}$$
    which corresponds to the vectors
    $$v_{11} = \big[ \begin{smallmatrix}
        1 \\ 0 \end{smallmatrix} \big], \quad  v_{12} = \big[ \begin{smallmatrix}
        0 \\ 1 \end{smallmatrix} \big], \quad  v_{21} = v_{22} = \big[ \begin{smallmatrix}
        0 \\ 0 \end{smallmatrix} \big],$$
    which, in turn, lead to the matrix
    $$B=\begin{bmatrix}
        1 & 1 \\ 0 & 0
    \end{bmatrix}.$$
    The error in \cite{GUTKIN201328} seems to stem from Lemma 1 being wrong, and thus equations (1) and (2) in that paper not being equivalent.

    \section{Generalizing the bracelet framework}\label{sec:bracelet}

    In this section we generalize, in the same spirit as the main Definition \ref{def:generalized-unistochastic}, the notions of \emph{bracelet conditions} and \emph{bracelet matrices}, ideas originating in \cite{au1979orthostochastic}. These notions play an important role in the study of unistochastic matrices, since bracelet conditions fully characterize the first non-trivial case, that of dimension $d=3$. Indeed, a $3 \times 3$ bistochastic matrix is unistochastic if and only if it satisfies the bracelet conditions \cite{au1979orthostochastic}, i.e.~iff it is a bracelet matrix.

    We first review the standard notion of bracelet condition. The intuitive idea behind it is that the elements of two rows of a unistochastic matrix corresponding to the same column cannot be too large simultaneously with respect to the other row elements. This is because the scalar product of the corresponding rows of the unitary matrix needs to be zero, so each individual term of the sum cannot be too large in magnitude with respect to the others. This intuition is encoded in the following definition:

    \begin{equation}\label{eq:def-bracelet-condition}
        \mathsf {Brac}_{d}:=\bigg\{ (\alpha,\beta)\in \Delta_d^2 \, : \, \forall i \in [d], \, \sqrt{\alpha_i \beta_i} \leq \sum\limits_{\substack{j=1 \\ j\neq i}}^d  \sqrt{\alpha_j \beta_j} \bigg\}.
    \end{equation}

    In the formula above, recall that the  $(d-1)$-dimensional probability simplex is the set of all probability vectors in $\mathbb R^d$:
    $$\Delta_d:=\Big\{\alpha \in \mathbb R^d \, : \, \forall i \in [d], \, \alpha_i\geq 0 \text{ and } \sum_{i=1}^d \alpha_i=1\Big\}.$$

    We have the following important definition; note that the term ``bracelet matrix / condition'' was introduced in \cite{rajchel2022algebraic}.

    \begin{definition}\label{def:bracelet-matrix}
        A bistochastic matrix $B \in \mathsf B_d$ is said to be a \emph{bracelet matrix} if all pairs of rows and all pairs of columns of $B$ satisfy the bracelet condition from \eqref{eq:def-bracelet-condition}. We introduce the set of bracelet matrices
        \begin{equation}\label{eq:def-bracelet-matrix}
            \mathsf L_d:=\Big\{B\in \mathsf B_d : \,  \forall i_1\neq i_2 \, (B_{i_1, \cdot}, B_{i_2, \cdot}) \in \mathsf{Brac}_d\, \text{ and }\, \forall j_1\neq j_2 \, (B_{\cdot, j_1}, B_{\cdot, j_2}) \in \mathsf{Brac}_d \big\}.
        \end{equation}

    \end{definition}

    It was observed in \cite{rajchel2022algebraic} that being bracelet is a necessary condition for unistochasticity. We give here the proof of this claim for the sake of completeness.

    \begin{proposition}
        \label{prop:U-subseteq-L}
        For all dimensions $d \geq 2$, we have $\mathsf U_d\subseteq \mathsf L_d$.
    \end{proposition}

    \begin{proof}Let $B = (B_{ij})_{i,j} \in \mathsf U_d$ be such that $B_{ij} = |U_{ij}|^2$, where $U \in \mathcal{U}(d)$ is a corresponding unitary matrix.
        Fix two row indices $i_1\neq i_2$.
        By unitarity,  for all $k \in [d]$ we have
        $$-U_{i_1k}\overline{U_{i_2k}}=
        \sum\limits_{\substack{j=1 \\ j\neq k}}^d U_{i_1j}\overline{U_{i_2j}}.$$
        Taking norm and applying the triangle inequality, for all $k \in [d]$ we obtain
        $$|U_{i_1k}\overline{U_{i_2k}}|\leq
        \sum\limits_{\substack{j=1 \\ j\neq k}}^d |U_{i_1j}\overline{U_{i_2j}}|,$$
        which translates into
        $$\sqrt{B_{i_1k}B_{i_2k}}\leq
        \sum\limits_{\substack{j=1 \\ j\neq k}}^d \sqrt{B_{i_1j}B_{i_2j}}.$$
        A similar computation shows that the columns of $B$ also satisfy the bracelet condition from Eq.~\eqref{eq:def-bracelet-condition}; hence, $B \in \mathsf L_d$, as claimed.
    \end{proof}

    The bracelet conditions characterize unistochasticity for $3 \times 3$ matrices (i.e., $\mathsf U_3 = \mathsf L_3$, see \cite{au1979orthostochastic}), while being only necessary for $d \geq 4$, see \cite{rajchel2022algebraic}. Note that the complete description of the non-convex set $\mathsf U_3$ was obtained, thanks to the characterization in terms of bracelet conditions, in \cite{nakazato1996set}. For example, for the bistochastic matrices studied in Figure \ref{fig:random-U3s}, which are of the form
    $$B = \begin{bmatrix}
        \lambda_1 & \lambda_2 & \lambda_3 \\
        \lambda_3 & \lambda_1 & \lambda_2 \\
        \lambda_2 & \lambda_3 & \lambda_1
    \end{bmatrix},$$
    the bracelet conditions read
    \begin{equation}\label{eq:bracelet-3}
        \sqrt{\lambda_i \lambda_j} \leq \sqrt{\lambda_i \lambda_k} + \sqrt{\lambda_j \lambda_k},
    \end{equation}
    for any permutation $(i,j,k)$ of the set $\{1,2,3\}$. The matrices satisfying these conditions are precisely the unistochastic matrices, and they correspond to the region delimited by the gray curves in Figure~\ref{fig:random-U3s}.

    \subsection{Generalized bracelet conditions}

    Since the idea behind the bracelet condition from Eq.~\eqref{eq:def-bracelet-condition} was to use the orthogonality of the rows/columns of a unitary operator, we generalize this insight to our setting in the following definition, encoding in it the block-orthogonality of block unitary matrices.

    \begin{definition}\label{def:generalized-bracelet-condition}
        A pair of probability vectors $\alpha,\beta \in \Delta_d$ is said to satisfy the \emph{generalized bracelet condition} of order $s$ if they correspond to the normalized squares of Frobenius norms of the blocks of some unitary matrix $U \in \mathcal {U}(d s)$:
        \begin{align} \label{eq:def-generalized-bracelet-condition}
            \nonumber \mathsf{Brac}_{d,s}:=\Big\{(\alpha,\beta)\in \Delta_d^2 \, : \, &\exists A_1,\ldots, A_d, B_1, \ldots, B_d \in \M{s} \text{ such that }\\[-0.33em]
            & \sum_{i=1}^d A_iA_i^*=\sum_{i=1}^d B_iB_i^*=I_s, \quad \sum_{i=1}^d A_iB_i^*=0_s, \\
            \nonumber &
            \tfrac{1}{s} \|A_i\|_F^2=\alpha_i
            \:  \text{ and } \
            \tfrac{1}{s}    \|B_i\|_F^2=\beta_i
            \quad \forall i \in [d]
            \Big\}.
        \end{align}
    \end{definition}
    The conditions above mean  precisely that we can expand the 2-row block matrix  $$
    \begin{bmatrix}
        A_1 & \cdots & A_d \\
        B_1 & \cdots & B_d
    \end{bmatrix}$$
    to a full unitary matrix $U\in \U{ds}$.
    Note that the case $d=1$ of the bracelet conditions is empty for every order, i.e.,  $\mathsf{Brac}_{1,s}=\emptyset$ for every $s \geq 1$; indeed, if $A_1A_1^*=B_1B_1^*=I_s$, then
    $A_1,B_1\in \U{s}$, which implies that $A_1B_1^*\neq 0$.

    We first show that the newly introduced conditions from Eq.~\eqref{eq:def-generalized-bracelet-condition} do indeed generalize the standard bracelet conditions from Eq.~\eqref{eq:def-bracelet-condition}.

    \begin{proposition}\label{prop:Bd1_is_Bd}
        For all dimensions $d \geq 2$, the generalized bracelet conditions of order $s=1$ are precisely the usual bracelet conditions: $\mathsf {Brac}_{d,1}=\mathsf {Brac}_d$.
    \end{proposition}
    \begin{proof} For $s=1$ the generalized bracelet condition takes the form
        $$
        \mathsf {Brac}_{d,1}=\left\{
        \begin{aligned}
            &(\alpha,\beta)\in \Delta_d^2 \, : \,
            \exists a_1,\ldots, a_d, b_1, \ldots, b_d \in \mathbb{C} \text{ such that} \\
            & \sum\nolimits_i |a_i|^2=\sum\nolimits_i |b_i|^2=1, \, \sum\nolimits_i a_i\bar b_i=0,\, \text{ and } \,
            |a_i|^2=\alpha_i, \,
            |b_i|^2=\beta_i \ \ \forall i\in [d]
        \end{aligned}\:
        \right\}.$$
        The inclusion $\mathsf {Brac}_{d,1}\subseteq \mathsf {Brac}_d$ follows by mimicking the    proof of  Proposition \ref{prop:U-subseteq-L}.
        The converse inclusion can be thought of   as a generalized version of the triangle inequality: for $l_1 \geq \ldots  \geq l_d \geq 0$ satisfying $ l_1 \leq \sum_{j=2}^d l_j$, there exist   $\theta_1,\ldots, \theta_d\in [0,2\pi)$ such that $\sum_{j=1}^d l_je^{\mathrm{i}\theta_j}=0$. Therefore, for any $(\alpha,\beta) \in \mathsf {Brac}_d$ we can choose the phases   $\theta_1,\ldots, \theta_d$ so that   $a_j := \sqrt{\smash[b]{\alpha_j}}e^{{\mathrm{i}\theta_j}}$ and $b_j := \sqrt{\smash[b]{\beta_j}}$ satisfy $\sum\nolimits_{j=1}^d a_j\bar b_j=0$. The other conditions follow trivially, and so   $(\alpha,\beta) \in \mathsf {Brac}_{d,1}$, as claimed.
    \end{proof}

    As it is the case for the sets $\mathsf{U}_{d,s}$ (see Proposition \ref{prop:s-convexity-Uds}), the sets $\mathsf{Brac}_{d,s}$ satisfy the following ``convexity'' relation:
    \begin{proposition}\label{prop:s-convexity-Bracds}
        For every   $d \geq 2$ and     $s,t \geq 1$, we have
        $\frac{s}{s+t}\mathsf {Brac}_{d,s}+\frac{t}{s+t}\mathsf {Brac}_{d,t}\subseteq \mathsf {Brac}_{d,s+t}.$
    \end{proposition}

    \begin{proof}
        Consider pairs of probability vectors $(\alpha, \beta) \in \mathsf{Brac}_{d,s}$ and $(\mu, \nu) \in \mathsf{Brac}_{d,t}$ together with the generating matrices $A_1,\ldots, A_d, B_1,\ldots, B_d \in \M{s}$ and $C_1,\ldots, C_d, D_1,\ldots, D_d \in \M{t}$, i.e.,
        $$\begin{bmatrix}
            A_1 & \cdots & A_d \\
            B_1 & \cdots & B_d
        \end{bmatrix}
        \xmapsto{\frac{1}{s}\| \cdot \|^2_F }
        \begin{bmatrix}
            \alpha \\
            \beta
        \end{bmatrix}
        =
        \begin{bmatrix}
            \alpha_1 & \cdots & \alpha_d \\
            \beta_1 & \cdots & \beta_d
        \end{bmatrix}
        $$
        $$\begin{bmatrix}
            C_1 & \cdots & C_d \\
            D_1 & \cdots & D_d
        \end{bmatrix}
        \xmapsto{\frac{1}{t}\| \cdot \|^2_F }
        \begin{bmatrix}
            \mu \\
            \nu
        \end{bmatrix}
        =
        \begin{bmatrix}
            \mu_1 & \cdots & \mu_d \\
            \nu_1 & \cdots & \nu_d
        \end{bmatrix}
        $$
        Then, using a direct sum construction, we have:
        $$\begin{bmatrix}
            A_1\oplus C_1 & \cdots & A_d\oplus C_d \\
            B_1\oplus D_1 & \cdots & B_d\oplus D_d
        \end{bmatrix}
        \xmapsto{\frac{1}{s+t}\| \cdot \|^2_F }
        \begin{bmatrix}
            \frac{s\alpha_1 + t\mu_1}{s+t} & \cdots & \frac{s\alpha_d + t\mu_d}{s+t}\\[0.33em]
            \frac{s\beta_1 + t\nu_1}{s+t} & \cdots & \frac{s\beta_d + t\nu_d}{s+t}
        \end{bmatrix}
        =
        \begin{bmatrix}
            \frac{s\alpha + t\mu}{s+t} \\[0.33em]
            \frac{s\beta + t\nu}{s+t}
        \end{bmatrix}.
        $$
        Hence, $\frac s{s+t}(\alpha, \beta) + \frac t{s+t}(\mu, \nu) \in \mathsf{Brac}_{d,s+t}$, as desired.
    \end{proof}

    The result above easily generalizes to more than two summands.

    \begin{corollary}
        For all dimensions $d \geq 2$ and all orders $s_1, \ldots, s_k \geq 1$, we have
        $$\frac{s_1}{s_1+\ldots+s_k}\mathsf {Brac}_{d,s_1}+\ldots+
        \frac{s_k}{s_1+\ldots+s_k}\mathsf {Brac}_{d,s_k}
        \subseteq \mathsf {Brac}_{d,s_1+\ldots+s_k}.$$
        In particular, for all $s \geq 1$ we have $\mathsf {Brac}_{d,1} =\mathsf {Brac}_{d} \subseteq \mathsf {Brac}_{d,s}$.
    \end{corollary}

    The generalized bracelet conditions on probability vectors introduced in \cref{def:generalized-bracelet-condition} are not easy to check in general, due to the fact that one needs to solve a quadratic problem in $s \times s$ matrices. We present next a necessary condition for a pair of probability vectors to satisfy the generalized bracelet conditions that is easily verifiable.

    \begin{proposition}
        For any pair of probability vectors satisfying the generalized bracelet condition $(\alpha,\beta) \in \mathsf{Brac}_{d,s}$, it holds that, for all $i \in [d]$ such that $\beta_i \geq 1-1/s$,
        $$\sqrt{\alpha_i \vphantom{\beta_i}}\sqrt{s \beta_i - (s-1)} \leq s \sum_{\substack{j=1\\j \neq i}}^d \sqrt{\alpha_j \beta_j}.$$
    \end{proposition}
    \begin{proof}
        The inequality is a simple consequence of the conditions on the matrices $A_i, B_j$ from \cref{def:generalized-bracelet-condition}. Start from
        $$-A_i B_i^* = \sum_{\substack{j=1\\j \neq i}}^d A_j B_j^*,$$
        where $i \in [d]$, and take the operator norm of both sides. For the left-hand side, we have the following lower bound ($\sigma_k(\cdot)$ denote below the singular values of a matrix, ordered decreasingly):
        $$\|-A_i B_i^*\| = \sigma_1(A_iB_i^*) \geq \sigma_1(A_i) \sigma_s(B_i),$$
        where we apply \cite[Eq.~(III.20)]{bhatia1997matrix}. Clearly, $\sigma_1(A_i) = \|A_i\| \geq s^{-1/2} \|A_i\|_F = \sqrt{\alpha_i}$. We also have
        $$\sigma_s(B_i)^2 = \|B_i\|_F^2 - \sum_{k=1}^{s-1} \sigma_k(B_i)^2 \geq s \beta_i - (s-1),$$
        where we have used the fact that that all the singular values of $B_i$ do not exceed $1$, which follows from $\sum_k B_k^{\vphantom{*}}B_k^* = I_s$.

        Moving now to the right-hand side, we have:
        $$\Big\|\sum_{\substack{j=1\\j \neq i}}^d A_j B_j^*\Big\| \leq \sum_{\substack{j=1\\j \neq i}}^d \|A_j\| \|B_j\| \leq \sum_{\substack{j=1\\j \neq i}}^d \|A_j\|_F \|B_j\|_F =  s\sum_{\substack{j=1\\j \neq i}}^d \sqrt{\alpha_j \beta_j},$$
        concluding the proof.
    \end{proof}
    Note that in the case $s=1$, the necessary conditions given in the Proposition above reduce to the usual bracelet conditions from \cref{eq:def-bracelet-condition}, exactly as the generalized bracelet conditions.

    \medskip

    In what follows, we analyze the sets $\mathsf {Brac}_{d,s}$ for different values of the dimension $d$ and of the generalization parameter $s$. We will show that for $d=2$ nothing new happens when  the value of $s$ changes, i.e., $\mathsf {Brac}_{2,s}=\mathsf {Brac}_{2}$ for every $s \geq 1$; however, for all dimensions $d \geq 3$ we will obtain strict inclusion $\mathsf{Brac}_{d,s}\supsetneq  \mathsf{Brac}_{d}$ for every $s > 1$. Let us start with an auxiliary lemma.

    \begin{lemma}\label{lem:exists-perm}
        Consider two real diagonal matrices $X=\diag(x_1,\cdots,x_s)$ and $Y=\diag(y_1,\cdots,y_s)$.
        The following conditions are equivalent:
        \begin{enumerate}
            \item[(i)] There exist $U,V\in \U{s}$  such that $X U Y =V$.
            \item[(ii)] There exists $\sigma\in\mathfrak S_s$  such that $|x_iy_{\sigma(i)}| = 1$ for every $i\in [s]$.
        \end{enumerate}
    \end{lemma}
    \begin{proof} We shall prove the double implication.\\
        $(i) \Leftarrow (ii)$: Let $\sigma\in\mathfrak S_s $ be such that  $|x_iy_{\sigma(i)}|=1 $ for all $i \in [s]$. Taking $U$ to be the phase-permutation matrix corresponding to $\sigma$ and phases $\operatorname{arg}(x_i y_{\sigma(i)})$, we have $XUY=I_s$.
        \\
        $(i) \Rightarrow (ii)$: If $XUY=V$ for some unitary matrices $U$ and $V$, then the real diagonal matrices $X,Y$ are both invertible; furthermore, we have
        $$I_s=V^*V=Y^*U^*X^*XUY=YU^*X^2UY,$$
        which implies that
        $$X^2U=UY^{-2}.$$
        The latter equation means that the $s$ linearly independent column vectors of the unitary matrix $U$ form $s$ eigenvectors of $X^2$, and so the $s$ eigenvalues of $X^2$ are exactly the $s$ diagonal elements of $Y^{-2}$. Therefore, the set of $s$ diagonal elements of $X^2$ (counting multiplicities) is equal to the set of $s$ diagonal elements of $Y^{-2}$. Hence, there exists $\sigma\in\mathfrak S_s$ such that $x_i^2 = 1/y_{\sigma(i)}^2$ for all $i \in [s]$, which concludes the proof.
    \end{proof}

    \begin{proposition}
        For $d=2$ and all orders $s \geq 1$, we have
        $$\mathsf {Brac}_{2,s}=\mathsf {Brac}_{2}= \left\{ \big( (p, 1-p), (1-p,p) \big) \, : \, p\in[0,1]\right\}.$$
    \end{proposition}

    \begin{proof}
        Fix $s \geq 2$ and let $\big((\alpha_1, 1-\alpha_1),  (\beta_1, 1-\beta_1)\big) \in \mathsf {Brac}_{2,s}$. It suffices to show that $\beta_1 = 1 - \alpha_1$.
        Let $A_1,A_2,B_1,B_2 \in \M{s}$ be such that
        $\tfrac 1s\Tr(A_1A_1^*) = \alpha_1$ and $\tfrac 1s\Tr(B_1B_1^*) = \beta_1$ as well as
        \begin{align*}
            A_1A_1^*+A_2A_2^*&=I_s\\
            B_1B_1^*+B_2B_2^*&=I_s\\
            A_1B_1^*+A_2B_2^*&=0_s.
        \end{align*}
        Since $A_1A_1^*$ and $A_2A_2^*$ commute, there exists $U\in \U{s}$ such that
        \begin{align*}
            A_1A_1^*&=U\diag(a_1,\ldots,a_s)U^*\\
            A_2A_2^*&=U\diag(1-a_1,\ldots,1-a_s)U^*
        \end{align*}
        for some real numbers $a_1, \ldots, a_s \in [0,1]$. Moreover, using the singular value decomposition, there exist $U_1, U_2 \in \U{s}$ such that
        \begin{align*}
            A_1 &= U \diag(\sqrt{a_1},\ldots,\sqrt{a_s})U_1^*\\
            A_2 &= U \diag(\sqrt{1-a_1},\ldots,\sqrt{1-a_s})U_2^*.
        \end{align*}
        Similarly, there exist $V, V_1,V_2 \in \U{s}$ such that
        \begin{align*}
            B_1 &= V \diag(\sqrt{\smash[b]{b_1}},\ldots,\sqrt{\smash[b]{b_s}}) V_1^*\\
            B_2 &= V \diag(\sqrt{\smash[b]{1-b_1}},\ldots,\sqrt{\smash[b]{1-b_s}}) V_2^*
        \end{align*}
        with $b_1, \ldots, b_s \in [0,1]$.
        Denoting $W_1:=U_1^*V_1$ and $W_2:=U_2^*V_2$,  we have $W_1, W_2 \in \U{s}$ and the condition  $A_1B_1^*+A_2B_2^*=0$ now reads
        \begin{align*}
            \diag(\sqrt{a_1},\ldots,\sqrt{a_s})W_1 \diag(&\sqrt{\smash[b]{b_1}},\ldots,\sqrt{\smash[b]{b_s}}) \\  =  &   -\diag(\sqrt{1-a_1},\ldots,\sqrt{1-a_s})W_2\diag(\sqrt{\smash[b]{1-b_1}},\ldots,\sqrt{\smash[b]{1-b_s}}).
        \end{align*}

        For non-degenerate $A_1,A_2,B_1,B_2$ (i.e., if none of $\alpha_i$'s or $\beta_i$'s equals $0$ or $1$), the condition  $A_1B_1^*+A_2B_2^*=0$   can be written as
        $$\diag\left(\sqrt{\frac{a_1}{1-a_1}},\ldots,\sqrt{\frac{a_s}{1-a_s}}\right)W_1 \diag\left(\sqrt{\frac{b_1}{1-b_1}},\ldots,\sqrt{\frac{b_s}{1-b_s}}\right)=-W_2.$$
        Using  \cref{lem:exists-perm}, there exists a permutation $\sigma\in \mathfrak S_s$  such that  $\frac{a_i}{1-a_i}\frac{b_{\sigma(i)}}{1-b_{\sigma(i)}}=1$, that is,  $a_i=1-b_{\sigma(i)}$, for every $i \in [s]$. Therefore,
        $$\alpha_1=\tfrac{1}{s}\Tr(A_1A_1^*)=\tfrac{1}{s}\sum_{i=1}^s a_i=\tfrac{1}{s}\sum_{i=1}^s (1-b_{\sigma(i)})=1-\beta_1;$$
        thus, for $p:=\alpha_1 \in [0,1]$ we have
        $$\begin{bmatrix}
            A_1 & A_2  \\
            B_1 & B_2
        \end{bmatrix}
        \xmapsto{\frac{1}{s}\| \cdot \|^2_F }
        \begin{bmatrix}
            p & 1-p  \\
            1-p  & p
        \end{bmatrix}.
        $$

        For general $A_1,A_2,B_1,B_2$, because the general linear group $\mathcal {GL}_s(\mathbb C)$ is dense in $\M{s}$ and the map $\frac{1}{s}\| \cdot \|^2_F$ is continuous,  we again obtain  $\alpha_1 = 1-\beta_1$, which finishes the proof.
    \end{proof}

    We now move on to the case $d=3$ and we show that $\mathsf{Brac}_{3,s}\supsetneq \mathsf {Brac}_{3,1} = \mathsf{Brac}_{3}$ for every $s \geq 2$. This fact will actually imply that
    $$\mathsf{Brac}_{d,s}\supsetneq \mathsf{Brac}_{d,1} = \mathsf{Brac}_{d} \ \text { for all }\ d \geq 3,\, s \geq 2$$ since for all $d,n,s$ we have
    $\mathsf{Brac}_{d+n,s} \cap \{ \alpha_{d+1} = \cdots = \alpha_{d+n} = \beta_{d+1} = \cdots = \beta_{d+n} = 0\} \cong \mathsf{Brac}_{d,s}.$

    Back to the claim about $d=3$, we shall focus on the slice
    \begin{align}
        \label{eq:slice-B3s}
        \mathsf {Brac}_{3,s}\cap \{\alpha_3=\beta_2=0\}=
        \big\{
        & \big( (\alpha_1, 1-\alpha_1, 0), (\beta_1, 0, 1-\beta_1) \big) \in \Delta_3^2 \, \colon  \exists A_1,A_2,B_1,B_3 \in \M{s} \text{ s.t.}
        \nonumber   \\[0.2em]
        &A_1A_1^*+A_2A_2^*=B_1B_1^*+B_3B_3^*=I_s,\, A_1B_1^*=0_s,
        \\[0.2em]
        \nonumber
        &\tfrac 1s\! \Tr(A_1A_1^*) = \alpha_1,\, \tfrac 1s \!\Tr(A_2A_2^*) = 1-\alpha_1,
        \\[0.2em] &
        \nonumber
        \tfrac 1s\! \Tr(B_1B_1^*) = \beta_1, \,\tfrac 1s \!\Tr(B_3B_3^*) = 1-\beta_1\,
        \big\}.
    \end{align}

    \begin{proposition}\label{prop:slice-es}
        For $d=3$ and every $s \geq 1$, we have
        $$\mathsf {Brac}_{3,s}\cap \{\alpha_3=\beta_2=0\}=\left\{ \big( (\alpha_1, 1-\alpha_1, 0), (\beta_1, 0, 1-\beta_1) \big) \in \Delta_3^2 \, : \, \lceil \alpha_1 s\rceil + \lceil \beta_1 s\rceil \leq s \right\}.
        $$
    \end{proposition}

    \begin{proof}
        Let $s \geq 1$. We shall prove the double inclusion.

        ``$\subseteq$'': Consider matrices $A_1,A_2, B_1, B_3$ as in Eq.~\eqref{eq:slice-B3s} and
        let  $\sqrt{a_1}\geq \ldots
        \geq \sqrt{a_s}$ be the singular values of $A_1$. The eigenvalues of $A_1A_1^*$ are therefore $a_1\geq \ldots \geq a_s \geq 0$, and $A_1A_1^* + A_2A_2^* = I_s$ guarantees that $a_1\leq 1$. Since $\rk(A_1)=\rk(A_1A_1^*)=\:\text{the number of nonzero } a_i$'s, it follows that $$\alpha_1 s = \Tr(A_1A_1^*)  = \sum_{i=1}^s a_i = \sum_{i=1}^{\rk(A_1)} a_i \leq \rk(A_1) a_1 \leq \rk(A_1).$$ In consequence,    $\lceil \alpha_1 s \rceil \leq \rk(A_1)$ and, similarly, $\lceil \beta_1 s \rceil \leq \rk(B_1)$. Finally, from    $A_1B_1^*=0$ we obtain
        $\rk B_1^* \leq \dim \ker A_1 = s - \rk A_1$, and so $\rk(A_1) + \rk(B_1) \leq s$, which proves the first inclusion.

        \smallskip

        ``$\supseteq$'':
        Conversely, given  $\alpha_1, \beta_1 \in [0,1]$ satisfying $\lceil \alpha_1 s\rceil + \lceil \beta_1 s\rceil \leq s$, let us consider
        \begin{align*}
            A_1&:=\diag\Big(\sqrt{a_1}, \ldots, \sqrt{\smash[b]{a_{\lceil \alpha_1s \rceil}}},0,\ldots,0\Big)\\
            A_2&:=\diag\Big(\sqrt{1-a_1}, \ldots, \sqrt{\smash[b]{1-a_{\lceil \alpha_1s \rceil}}},1,\ldots,1\Big)\\
            B_1&:=\diag\Big(0,\ldots,0,\sqrt{\smash[b]{b_1}},\ldots,\sqrt{\smash[b]{b_{\lceil \beta_1s \rceil}}}\Big)\\
            B_3&:=\diag\Big(1,\ldots,1,\sqrt{\smash[b]{1-b_1}},\ldots,\sqrt{\smash[b]{1-b_{\lceil \beta_1s \rceil}}}\Big),
        \end{align*}
        where
        $$a_1 = \cdots = a_{\lfloor \alpha_1 s \rfloor} = 1$$
        and, in the case when $\alpha_1 s$ is not an integer,
        $$a_{\lceil \alpha_1s \rceil} = a_{\lfloor \alpha_1 s \rfloor + 1} = \alpha_1 s - \lfloor \alpha_1 s \rfloor \in (0,1).$$
        The $b_i$'s are defined analogously, using $\beta_1 s$.
        The assumption $\lceil \alpha_1 s\rceil + \lceil \beta_1 s\rceil \leq s$ guarantees that the non-zero elements of $A_1$ and $B_1$ do not overlap, which implies that $A_1B_1^* = 0_s$. One can easily verify that all the other conditions from Eq.~\eqref{eq:slice-B3s} hold as well, and so the proof is finished.
    \end{proof}

    Let us consider
    $$E(s):=\left\{(\alpha_1, \beta_1) \in [0,1]^2 \colon \lceil \alpha_1 s\rceil + \lceil \beta_1 s\rceil \leq s \right\}.$$
    The sets $E(2), E(3), E(4), E(5)$ are displayed in Figure \ref{fig:Es}. Note that every $E(s)$ contains  the  axes:
    $$E(1) = \big\{ (\alpha_1, 0) \, : \, \alpha_1 \in [0,1]\big\} \cup \big\{(0,\beta_1) \, : \, \beta_1 \in [0,1]\big\} \subseteq E(s) \ \text { for every}\ s \geq 1.$$
    Moreover, we have
    $$
    \big\{ (\alpha_1,\beta_1) \, : \, \alpha_1+\beta_1\leq 1-  {2}/{s}\big\}\subseteq
    E(s)\subseteq
    \big\{{ (\alpha_1,\beta_1) \, : \, \alpha_1+\beta_1\leq 1\big\}} \ \text { for every}\ s \geq 1;
    $$
    hence,
    $$ \overline{\lim_{s \to \infty} E(s)} = \big\{ { (\alpha_1,\beta_1) \in [0,1]^2 \, : \, \alpha_1+\beta_1\leq 1\big\}}.$$
    This shows that, after taking the limit $s \to \infty$ and the closure, the bracelet conditions corresponding to the slice considered in \cref{eq:slice-B3s} become trivial; this result is in the spirit of \cref{thm:union-is-all-bistochastic}.

    \begin{figure}[htb]
        \centering
        \includegraphics[scale=.55]{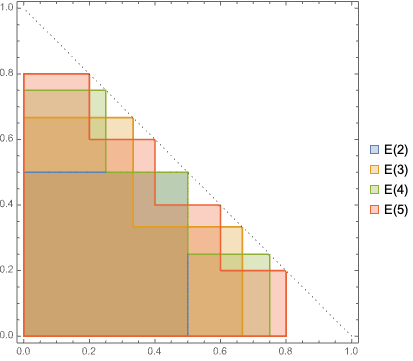}
        \caption{The (interior of the) sets $E(s)$ for $s=2,3,4,5$.}
        \label{fig:Es}
    \end{figure}

    \subsection{Generalized bracelet matrices}

    Similarly to Definition \ref{def:bracelet-matrix}, we introduce the set of generalized bracelet matrices.

    \begin{definition}\label{def:generalized-bracelet-matrix}
        A bistochastic matrix $B \in \mathsf B_d$ is called a \emph{generalized bracelet matrix} of order $s$ if all pairs of rows and all pairs of columns of $B$ satisfy the generalized bracelet condition of order $s$ from Eq.~\eqref{eq:def-generalized-bracelet-condition}. That is, the set of generalized bracelet matrices is defined as
        \begin{equation}\label{eq:def-generalized-bracelet-matrix}
            \mathsf L_{d,s}:=\big\{B\in \mathsf B_d : \,  \forall i_1\neq i_2 \, (B_{i_1 \cdot}, B_{i_2 \cdot}) \in \mathsf{Brac}_{d,s} \text{ and } \forall j_1\neq j_2 \, (B_{\cdot j_1}, B_{\cdot j_2}) \in \mathsf{Brac}_{d,s} \big\}.
        \end{equation}
    \end{definition}

    The following result is a generalization of \cref{prop:U-subseteq-L}:
    \begin{proposition}
        For all dimension $d \geq 2$ and all order $s \geq 1$ we have
        $$\mathsf U_{d,s} \subseteq \mathsf L_{d,s} \subseteq \mathsf B_d.$$
        In particular, the closure of the set of all generalized bracelet matrices is the full Birkhoff polytope:    $$\overline{\bigcup_{s\geq 1}\mathsf L_{d,s}}=\mathsf B_d.$$
    \end{proposition}
    \begin{proof}
        The only fact that needs checking is
        $\mathsf U_{d,s} \subseteq \mathsf L_{d,s}$, as then the claim about $\bigcup_{s\geq 1}\mathsf L_{d,s}$  will follow from \cref{thm:union-is-all-bistochastic}.
        Fix $B \in \mathsf U_{d,s}$ and let $U \in \mathcal U(ds)$ be the corresponding unitary matrix, i.e., $\frac 1s \operatorname{Tr}(U_{ij}U_{ij}^*) = B_{ij}$ for all $i,j \in [d]$. Using the unitarity of $U$, one easily verifies that the pair of different rows $(B_{i_1\cdot}, B_{i_2\cdot})$   satisfies the generalized bracelet condition $\mathsf{Brac}_{d,s}$ with matrices $U_{i_11}, \ldots, U_{i_1d}; U_{i_21}, \ldots, U_{i_2d} \in \M{s}$. Analogously for columns; hence $B \in \mathsf L_{d,s}$, as claimed.
    \end{proof}

    We show next that the (non-convex) sets of generalized bracelet matrices $\mathsf L_{d,s}$ have a very intricate inclusion structure. To this end, we study the intersections of $\mathsf L_{d,s}$ and of $\mathsf U_{d,s}$ with segments connecting two extremal points of the Birkhoff polytope.

    \begin{proposition}\label{intersection}
        Let $d \geq 3$ and consider permutations $\pi, \sigma \in \mathfrak S_d$ such that $\pi^{-1}\sigma$ has a $p$-cycle for some $p \geq 3$. Then the intersection of the segment $[\pi, \sigma]$ with  $\mathsf U_{d,s}$ is equal to the intersection of  $[\pi, \sigma]$ with  $\mathsf L_{d,s}$,  and coincides with the discrete set  of convex mixtures of $\pi$ and $\sigma$ with rational weights with denominator $s$:
        $$\{\lambda \in [0,1] \, : \, (1-\lambda) \pi + \lambda \sigma \in \mathsf U_{d,s}\} = \{\lambda \in [0,1] \, : \, (1-\lambda) \pi + \lambda \sigma \in \mathsf L_{d,s}\} = \{{k}/{s}\colon k = 0, \ldots, s\}.$$
        In particular, this holds for the $1$-faces (i.e.~edges) of the Birkhoff polytope $\mathsf B_d$ (for which $\pi^{-1}\sigma$ is a $d$-cycle).
    \end{proposition}
    \begin{proof} The inclusion of the first set in the second one is trivial.
        We shall prove that the second set is contained in the third, and then that the third set is contained in the first.

        For the first inclusion, without loss of generality  we may assume that the decomposition of $\pi^{-1}\sigma$ contains the cycle $(1,2,\ldots,p)$ for some $p \geq 3$. Furthermore, we may also assume that $\pi$ is the identity permutation; hence, $(1-\lambda) \pi + \lambda \sigma$ has the form:
        $$
        \begin{bmatrix}
            1-\lambda & \lambda  & 0 & \cdots & 0 & 0\\
            0 & 1- \lambda  & \lambda & \cdots & 0 & 0\\
            \vdots & \vdots & \vdots & \ddots & \vdots & \vdots\\
            0 & 0 & 0 & \cdots & 1-\lambda & \lambda\\
            \lambda & 0 & 0 & \cdots & 0 & 1-\lambda
        \end{bmatrix}_{p \times p}
        \bigoplus\  \Big[ \text{another matrix of size } (d-p) \times (d-p)\Big].
        $$
        Considering  any two rows of the $p \times p$ submatrix, we obtain, after permuting the columns, the slice analysed in Prop. \ref{prop:slice-es}:
        if $(1-\lambda) \pi + \lambda \sigma \in \mathsf L_{d,s}$, there exist $A_1, A_2, B_1, B_3\in \M{s}$ such that
        $$\begin{bmatrix}
            A_1 & A_2 & 0 & 0 & \cdots \\
            B_1 & 0 & B_3 & 0 & \cdots
        \end{bmatrix}_{2s \times ps}
        \xmapsto{\frac{1}{s}\| \cdot \|^2_F }
        \begin{bmatrix}
            \lambda & 1-\lambda & 0 & 0 & \cdots \\
            1-\lambda & 0 & \lambda & 0 & \cdots
        \end{bmatrix}_{2 \times p}.
        $$

        As in the  ``$\subseteq$'' part of the proof of Prop. \ref{prop:slice-es}, we deduce that
        $\rk(A_1)\geq  \lceil \lambda s \rceil$ and $\rk(B_1)\geq \lceil (1-\lambda) s \rceil$, and then
        $\lceil \lambda s\rceil + \lceil (1-\lambda) s\rceil \leq s $; thus, $\lambda = {k}/{s}$ for some $k\in \{0,1,\ldots, s\}$, proving the first claim.

        As for the second inclusion, use Corollary \ref{Corollary}, together with the fact that permutation matrices are unistochastic, to conclude that
        $$ \frac{s-k}{s}\pi+\frac{k}{s}\sigma\in \mathsf U_{d,s}.$$

        Finally, the claim about the face structure of the Birkhoff polytope can be found in, e.g., \cite[Theorem 2.2]{brualdi1977convex}.
    \end{proof}

    \begin{corollary}\label{cor:not-increasing-in-s}
        The family of sets $\mathsf U_{d,s}$ is not increasing in $s$.
    \end{corollary}

    Note that the result above was the motivation behind the study of the slice of the bracelet condition set considered in Eq.~\eqref{eq:slice-B3s}. Indeed, if we consider the segment $[(1)(2)(3), (123)]$ connecting the identity and the full cycle permutations of $\mathfrak S_3$, we have
    $$(1-\lambda) \cdot (1)(2)(3) + \lambda \cdot (123) =
    \begin{bmatrix}
        1-\lambda & \lambda & 0 \\
        0 & 1-\lambda & \lambda\\
        \lambda & 0 & 1-\lambda
    \end{bmatrix}.
    $$
    We see that any two rows (or columns) of the matrix above have the zero pattern found in Eq.~\eqref{eq:slice-B3s}.

    As a final remark, we show that the  blue line in \cref{fig:triangle-star-dot}, i.e.~the set of all bistochastic matrices of the form
    $$B = \begin{bmatrix}
        x & y & y \\
        y & x & y \\
        y & y & x
    \end{bmatrix},$$ is a subset of $\mathsf U_{3,2}$.
    To this end, we consider a variable $q \in [0,1]$ and  define
    $$w_\pm = \tfrac 1 2 \big( 1-q \pm \sqrt{1+2q-3q^2}\,\big),$$
    and
    $$U_q := \begin{bNiceMatrix}[c,margin]
        \CodeBefore
        \rectanglecolor{blue!10}{1-1}{2-2}
        \rectanglecolor{red!10}{1-3}{2-4}
        \rectanglecolor{blue!10}{1-5}{2-6}
        \rectanglecolor{red!10}{3-1}{4-2}
        \rectanglecolor{blue!10}{3-3}{4-4}
        \rectanglecolor{red!10}{3-5}{4-6}
        \rectanglecolor{blue!10}{5-1}{6-2}
        \rectanglecolor{red!10}{5-3}{6-4}
        \rectanglecolor{blue!10}{5-5}{6-6}
        \Body
        q & 0 & w_- & 0 & w_+ & 0\\
        0 & q & 0 & w_+ & 0 & w_-\\
        w_+ & 0 & q & 0 & w_- & 0\\
        0 & w_- & 0 & q & 0 & w_+\\
        w_- & 0 & w_+ & 0 & q & 0\\
        0 & w_+ & 0 & w_- & 0 & q
    \end{bNiceMatrix} \stackrel{\phi_{3,2}}{ \longmapsto }
    \begin{bNiceMatrix}[c,margin]
        \CodeBefore
        \chessboardcolors{blue!10}{red!10}
        \Body
        q^2 & \frac{1}{2}(1-q^2) & \frac{1}{2}(1-q^2)\\[0.33em]
        \frac{1}{2}(1-q^2) & q^2 &\frac{1}{2}(1-q^2)\\[0.33em]
        \frac{1}{2}(1-q^2) & \frac{1}{2}(1-q^2) & q^2
    \end{bNiceMatrix} \in \mathsf U_{3,2}.$$
    One can easily check by direct computation that  the matrix $U_q$   is unitary (actually orthogonal, see next section) and that the associated bistochastic matrices fill in the blue line in \cref{fig:triangle-star-dot}. It is interesting to notice that $U_q$ is equal to a direct sum $U_q= U_{135} \oplus U_{246}$ of two circulant matrices acting on odd, resp.~even, indices. Note that the blue point at the bottom end of the blue line (i.e.,~the matrix $\frac 12 (P_{(132)} + P_{(123)})$), corresponding to $q=0$, has also been discussed in \cref{ex:generalized-but-not-unistochastic}, where a different $6 \times 6$ orthogonal matrix has been used to show that it is 2-unistochastic.

    Finally, consider the red point in \cref{fig:triangle-star-dot}, corresponding to the convex combination, with weights 2/3 and 1/3, respectively, of the blue and green points. It corresponds to the bistochastic matrix
    $$B = \frac{2}{3} \begin{bmatrix}
        0 & 1/2 & 1/2 \\
        1/2 & 0 & 1/2 \\
        1/2 & 1/2 & 0
    \end{bmatrix} + \frac{1}{3}\begin{bmatrix}
        1/9 & 4/9 & 4/9 \\
        4/9 & 1/9 & 4/9 \\
        4/9 & 4/9 & 1/9
    \end{bmatrix} = \begin{bmatrix}
        1/27 & 13/27 & 13/27 \\
        13/27 & 1/27 & 13/27 \\
        13/27 & 13/27 & 1/27
    \end{bmatrix}
    $$
    Since the green point lies on the hypocycloid curve: $\sqrt{4/9 \cdot 4/9} = 2 \sqrt{1/9 \cdot 4/9}$,  it corresponds to a unistochastic matrix. The blue point corresponds, as shown above, to a 2-unistochastic matrix. Hence, in virtue of \cref{prop:s-convexity-Uds}, $B$ is 3-unistochastic.
    This disproves a bistochastic version of  \cite[Conj.~7.1]{shahbeigi2021log}, which suggested that the set of 3-unistochastic matrices restricted to the simplex in \cref{fig:triangle-star-dot} coincides with the union of the region delimited by the hypocycloid and the yellow  Star of David shape generated by the two triangles.

    \begin{figure}
        \centering
        \includegraphics[width=0.5\textwidth]{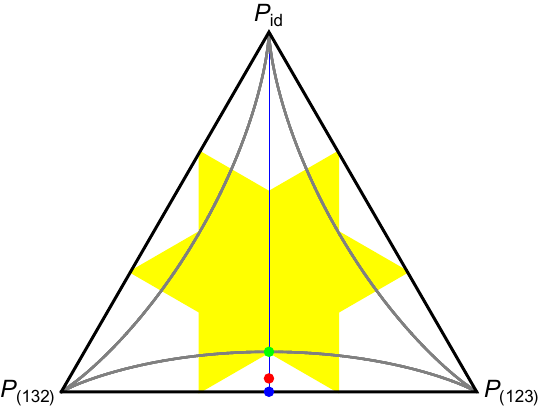}
        \caption{A slice through the Birkhoff polytope $\mathsf B_3$  corresponding to the simplex generated by the permutation matrices $P_{\mathrm{id}}$, $P_{(123)}$, and $P_{(132)}$; cf.  Figure \ref{fig:random-U3s} and \cite[Figure 7(b)]{shahbeigi2021log}.}
        \label{fig:triangle-star-dot}
    \end{figure}

    \section{Generalized orthostochastic matrices}\label{sec:orthostochastic}

    Much of the theory developed in the previous sections for generalized unistochastic matrices can be carried out to the case of \emph{generalized orthostochastic} matrices, which is what we do in this section. However, since many things are very similar, we shall only present the main definitions and some observations. We leave the detailed study of generalized orthostochastic matrices (and that of their quaternionic counterpart, the \emph{qustochastic} matrices) to future work.

    Recall that the function $\phi_{d,s}$ from \cref{def:generalized-unistochastic} maps a $d \times d$ block-matrix (with blocks of size $s \times s$) to the matrix of normalized squares of Frobenius norms of the blocks. We denote by $\mathcal O(n)$ the group of $n \times n$ orthogonal matrices.

    \begin{definition}
        We define
        $$\mathsf O_{d,s}:=\phi_{d,s}(\mathcal O(ds))$$ to be the set of \emph{generalized orthostochastic matrices}.
    \end{definition}

    As in the complex case, for $s=1$ we recover the usual orthostochastic matrices, which have received, along with the unistochastic matrices, a lot of attention in the literature \cite{au1979orthostochastic,au1991permutation,bengtsson2005birkhoff,zyczkowski2003random,chterental2008orthostochastic}. Clearly, $\mathsf O_{d,s}\subseteq\mathsf U_{d,s}$, with the inclusion being strict for $d \geq 3$.

    Importantly, the van der Waerden matrix $J_d/d$ is orthostochastic if and only if there exists a~real Hadamard matrix of order $d$ \cite{hadamard1893resolution,hedayat1978hadamard,kharaghani2005hadamard}. This can only happen if $d=2$ or if $d$ is a multiple of four, and it has long been conjectured that these conditions are also sufficient. In particular, the distance between $J_3/3$ and the set $\mathsf O_{d,1}$ is equal to $\sqrt 2/3$ \cite[Proposition 3.2]{chterental2008orthostochastic}.

    We prove now the main result of this section.

    \begin{proposition}
        For all dimensions $d \geq 2$ and all orders $s \geq 1$,   we have  $$\mathsf U_{d,s}\subseteq\mathsf O_{d,2s}.$$
    \end{proposition}

    \begin{proof}
        The result follows from the standard embedding $\mathcal U(n) \subseteq \mathcal O(2n)$, obtained by replacing a~complex entry $z_{ij}$ by the $2 \times 2$ block $\big[ \begin{smallmatrix}\operatorname{Re} z_{ij} & \operatorname{Im} z_{ij} \\
            -\operatorname{Im} z_{ij} & \operatorname{Re} z_{ij} \end{smallmatrix}\big]$.
    \end{proof}

    \begin{corollary}
        For all $d \notin \{2\} \cup 4\mathbb N$, we have
        $$J_d/d \in \mathsf O_{d,2} \setminus \mathsf O_{d,1}.$$
    \end{corollary}
    \begin{proof}
        We have $J_d/d \in \mathsf U_{d,1} \subseteq \mathsf O_{d,2}$. On the other hand, since there cannot exist a real Hadamard matrix of order $d$, we have $J_d/d \notin \mathsf O_{d,1}$, as claimed.
    \end{proof}

    \section{Random generalized unistochastic matrices}\label{sec:random}

    In the previous sections we have introduced and discussed generalized unistochastic matrices, which form the set
    $$\mathsf U_{d,s} = \phi_{d,s}(\U{ds}),$$
    where $\phi_{d,s}$ is the map from Eq.~\eqref{eq:def-phi-d-s}. As the unitary group $\U{ds}$ comes equipped with the (normalized) Haar measure $\mathfrak h_{ds}$, it is natural to introduce and examine its image measure via~$\phi_{d,s}$.

    \begin{definition}\label{def:mu-ds}
        We endow the set of generalized unistochastic matrices $\mathsf U_{d,s}$ with the probability measure
        \begin{equation*}
            \mu_{d,s} = (\phi_{d,s})_\# \mathfrak h_{ds},
        \end{equation*}
        that is, the image measure of the normalized Haar distribution $\mathfrak h_{ds}$ on $\U{ds}$ through the map $\phi_{d,s}$. In other words, if $U \in \U{ds}$ is Haar-distributed, then $B:=\phi_{d,s}(U) \in \mathsf U_{d,s}$ is $\mu_{d,s}$-distributed.
    \end{definition}

    We recall the following result about the first few joint moments of the entries of a Haar-distributed random unitary matrix.

    \begin{lemma}[{{\cite[Proposition 4.2.3]{hiai2000semicircle}}}]\label{lem:moments-unitary}
        Let $U=(U_{ij})_{ i,j \in [n]} \in \U{n}$ be Haar-distributed. We have
        \begin{align*}
            \mathbb{E}\Big[|U_{ij}|^2\Big]&=\frac{1}{n}\qquad (1\leq i,j\leq n)\\
            \mathbb{E}\Big[|U_{ij}|^4\Big]&=\frac{2}{n(n+1)}\qquad (1\leq i,j\leq n)\\
            \mathbb{E}\Big[|U_{ij}|^2|U_{i'j}|^2\Big]=\mathbb{E}\Big[|U_{ij}|^2|U_{ij'}|^2\Big]
            &=\frac{1}{n(n+1)}\qquad (i\neq i',j\neq j')\\
            \mathbb{E}\Big[|U_{ij}|^2|U_{i'j'}|^2\Big]
            &=\frac{1}{n^2-1}\qquad (i\neq i',j\neq j').
        \end{align*}
    \end{lemma}

    We leverage now this result to obtain the first moments of a random generalized unistochastic matrix.

    \begin{proposition}\label{prop:first-moments-B}
        For a Haar-distributed random unitary matrix $U\in \U{ds}$, consider the corresponding $\mu_{d,s}$-distributed bistochastic matrix
        $$B:=\phi_{d,s}(U) = \left(\frac{1}{s}||U_{ij}||_F^2\right)_{i,j \in [d]} \in \mathsf B_d.$$ For all $1 \leq i\neq i',j\neq j' \leq d$ and $n:=ds$ we have:
        \begin{align*}
            \mathbb{E}\Big[B_{ij}\Big]&=\frac{1}{d}\\
            \mathbb{E}\Big[B_{ij}^2\Big]&=\frac{d(s^2+1)-2}{d(n^2-1)} \\
            \mathbb{E}\Big[B_{ij}B_{i'j}\Big]=\mathbb{E}\Big[B_{ij}B_{ij'}\Big]&=
            \frac{ds^2-1}{d(n^2-1)}\\
            \mathbb{E}\Big[B_{ij}B_{i'j'}\Big]&=
            \frac{s^2}{n^2-1}.
        \end{align*}

    \end{proposition}

    \begin{proof}
        We show the different claims one by one, using Lemma \ref{lem:moments-unitary}.
        $$\mathbb{E}\Big[B_{ij}\Big]
        =\frac{1}{s}\mathbb{E}\Big[\Tr(U_{ij}U_{ij}^*)\Big]
        =\frac{1}{s}\mathbb{E}\Big[||U_{ij}||_F^2\Big]
        =\frac 1 s \sum_{k,l=1}^s\mathbb{E}\Big[|U_{ij}(k,l)|^2\Big]
        =\frac{1}{s}\cdot s^2\cdot \frac{1}{n}
        =\frac{1}{d}.$$

        \begin{align*}
            \mathbb{E}\Big[B_{ij}^2\Big]
            &=\frac{1}{s^2}\mathbb{E}\Big[\Tr(U_{ij}U_{ij}^*)^2\Big]
            =\frac{1}{s^2}\mathbb{E}\Big[\big(\sum_{k,l=1}^s|U_{ij}(k,l)|^2\big)\big(\sum_{p,q=1}^s|U_{ij}(p,q)|^2\big)\Big]\\
            &=\frac{1}{s^2}\mathbb{E}\Big[\big(\sum_{k,l=1}^s|U_{ij}(k,l)|^2\big)\big(|U_{ij}(k,l)|^2+\sum_{\substack{q=1\\q\neq l}}^s|U_{ij}(k,q)|^2+\sum_{\substack{p=1\\p\neq k}}^s|U_{ij}(p,l)|^2+\sum_{\substack{p,q=1\\p\neq k,\, q\neq l}}^s\!\!\!|U_{ij}(p,q)|^2\big)\Big]\\
            &=\frac{1}{s^2}\cdot s^2\cdot \Big[\frac{2}{n(n+1)}+2(s-1)\frac{1}{n(n+1)}+(s-1)^2\frac{1}{n^2-1}\Big]=\frac{d(s^2+1)-2}{d(n^2-1)}.
        \end{align*}

        \begin{align*}
            \mathbb{E}\Big[B_{ij}B_{ij'}\Big]
            &=\frac{1}{s^2}\mathbb{E}\Big[\big(\sum_{k,l=1}^s|U_{ij}(k,l)|^2\big)\big(\sum_{p,q=1}^s|U_{ij'}(p,q)|^2\big)\Big]\\
            &=\frac{1}{s^2}\mathbb{E}\Big[\big(\sum_{k,l=1}^s|U_{ij}(k,l)|^2\big) \big(\sum_{q=1}^s|U_{ij'}(k,q)|^2+\sum_{\substack{p,q=1\\p\neq k}}^s|U_{ij'}(p,q)|^2\big)\Big]\\
            &=\frac{1}{s^2}\cdot s^2\cdot \Big[s\frac{1}{n(n+1)}+s(s-1)\frac{1}{n^2-1}\Big].
        \end{align*}

        \begin{align*} \mathbb{E}\Big[B_{ij}B_{i'j'}\Big]
            =\frac{1}{s^2}\mathbb{E}\Big[\big(\sum_{k,l=1}^s|U_{ij}(k,l)|^2\big)\big(\sum_{p,q=1}^s|U_{i'j'}(p,q)|^2\big)\Big]
            =\frac{1}{s^2}\cdot s^4\cdot \frac{1}{n^2-1}=\frac{s^2}{n^2-1}.       \end{align*}
    \end{proof}

    \begin{remark}
        We could get the same results using the (graphical) Weingarten calculus \cite{Collins_2006,Collins_2010} used to compute general integrals over the unitary group with respect to the Haar measures.
    \end{remark}

    Let us now analyze the correlations between different matrix elements of $B$. Recall that  \emph{Pearson's correlation coefficient} of a pair of random variables $(X,Y)$ is defined as
    $$\rho(X,Y) := \frac{\operatorname{Cov}(X,Y)}{\sqrt{\operatorname{Var}(X)}\cdot \sqrt{\operatorname{Var}(Y)}}.$$

    \begin{corollary}
        For all $d \geq 2$, $s \geq 1$, $n:=ds$, and $ 1 \leq i\neq i', j\neq j' \leq d$, we have:
        \begin{align*}
            \operatorname{Var}(B_{ij})&=\frac{(d-1)^2}{d^2(n^2-1)}
            \\
            \operatorname{Cov}(B_{ij},B_{i'j})=\operatorname{Cov}(B_{ij},B_{ij'})&=-\frac{d-1}{d^2(n^2-1)}\\
            \rho(B_{ij},B_{i'j})=\rho(B_{ij},B_{ij'})&=-\frac{1}{d-1}
            \\
            \operatorname{Cov}(B_{ij},B_{i'j'})&=\frac{1}{d^2(n^2-1)}
            \\
            \rho(B_{ij},B_{i'j'})&=\frac{1}{(d-1)^2}.
        \end{align*}
    \end{corollary}
    \begin{proof}
        This follows by direct computation from Prop. \ref{prop:first-moments-B}; the details are left to the reader.
    \end{proof}

    \begin{remark}
        Note that the covariance (and the correlation coefficient) of elements of $B$ situated on the same row (or column) is negative; this anti-correlation is explained by the normalization conditions of bistochastic matrices. However, the correlation of elements not belonging to the same row and column is positive.

        Let us also note that while the covariance of the matrix elements of $B$ decreases (in absolute value) with the parameter $s$, the correlation coefficient is constant (at fixed matrix dimension $d$).
    \end{remark}

    The fact that the variance of $B_{ij}$ decreases with the parameter $s$ at fixed matrix dimension $d$ means that the distribution $\mu_{d,s}$ concentrates, as $s \to \infty$, around the van der Waerden matrix. The same phenomenon can be seen at the level of spectra, see Figure \ref{fig:random-spectra}: the non-trivial eigenvalues of the random matrices $B \sim \mu_{d,s}$ tend to concentrate around the origin as $s$ grows. We point the reader interested in spectral properties of bistochastic and unistochastic matrices to the papers \cite{zyczkowski2003random,cappellini2009random}. Finally, note the pair of complex eigenvalues outside the gray-bounded region in the middle top panel of Figure \ref{fig:random-spectra}; they are a signature of the fact that $\mathsf U_{3,2} \supsetneq \mathsf U_{3,1}$, see Example~\ref{ex:generalized-but-not-unistochastic}.

    \begin{figure}[htb]
        \centering
        \includegraphics[width=0.275\textwidth]{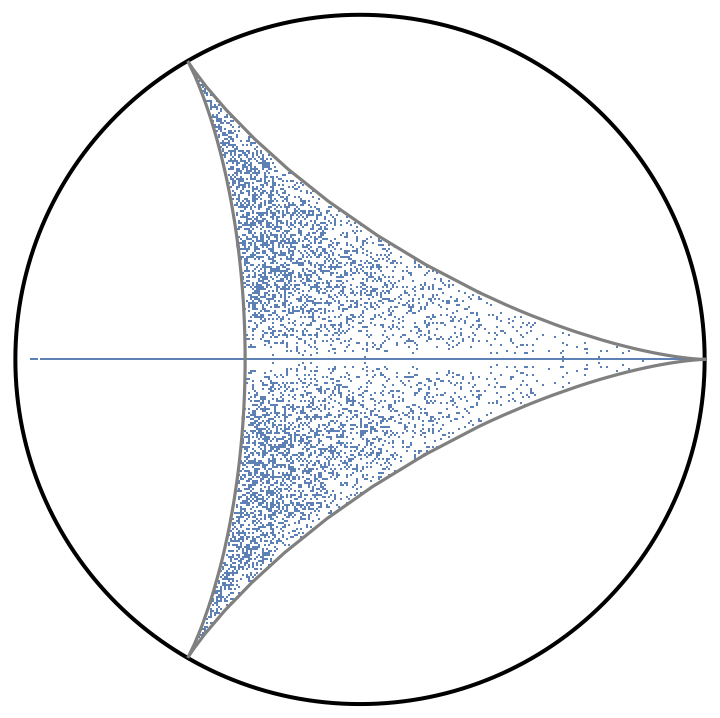} \qquad
        \includegraphics[width=0.275\textwidth]{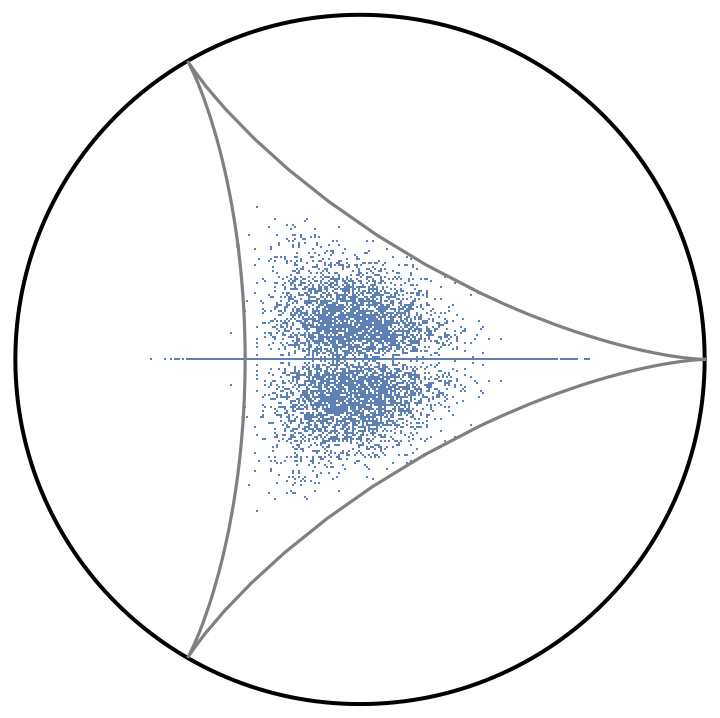} \qquad
        \includegraphics[width=0.275\textwidth]{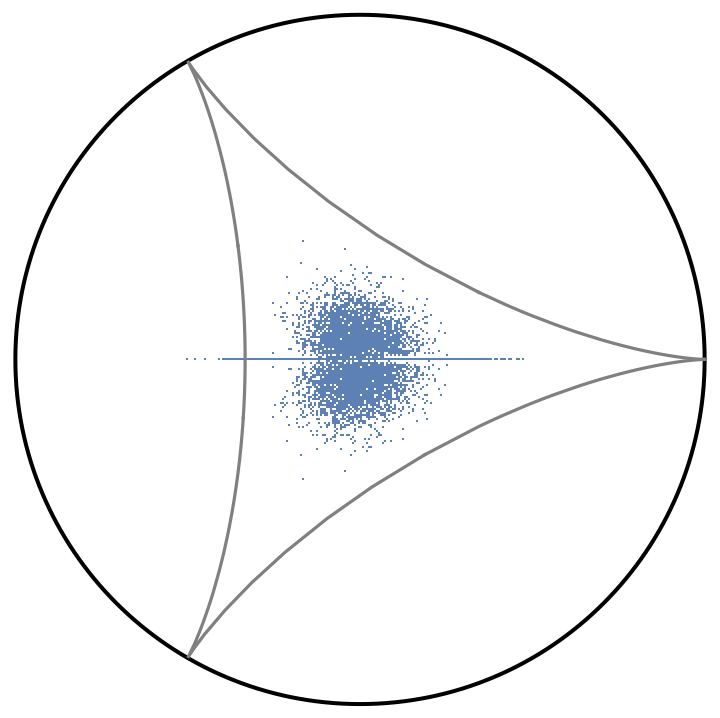}\\
        \includegraphics[width=0.275\textwidth]{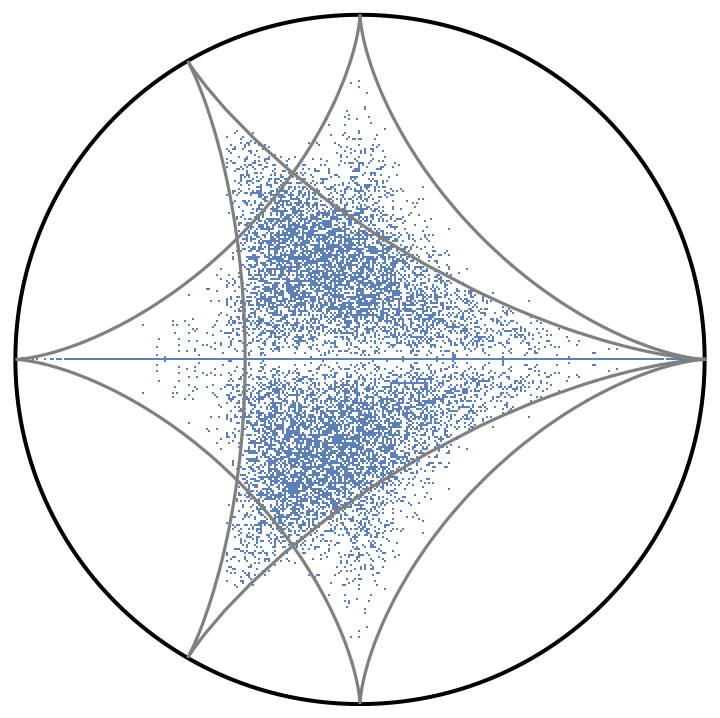} \qquad
        \includegraphics[width=0.275\textwidth]{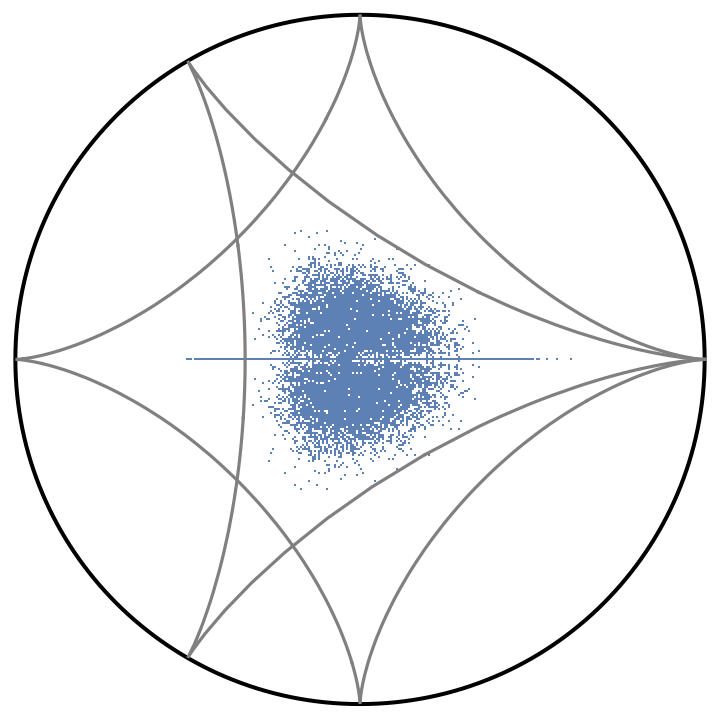} \qquad
        \includegraphics[width=0.275\textwidth]{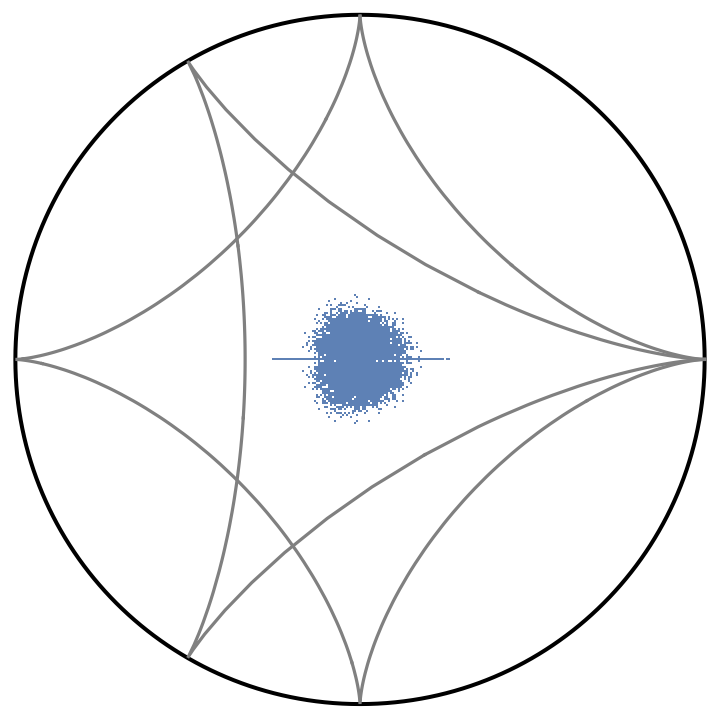}
        \caption{Spectra of random generalized bistochastic matrices. We plot the (complex) eigenvalues of 10\,000 samples from the measure $\mu_{d,s}$ introduced in Def.~\ref{def:mu-ds}. On the top row, we have $d=3$, and, respectively, $s=1,2,3$. On the bottom row, $d=4$ and $s=1,2,4$. We also plot in gray the hypocycloid curves which are conjectured to bound the spectra in the unistochastic case $s=1$, see \cite[Section 4.3]{zyczkowski2003random}.}
        \label{fig:random-spectra}
    \end{figure}

    \noindent\textbf{Acknowledgements.} We thank Karol {\.Z}yczkowski for bringing the paper \cite{shahbeigi2021log} to our attention. I.N.~was supported by the ANR projects \href{https://esquisses.math.cnrs.fr/}{ESQuisses}, grant number ANR-20-CE47-0014-01 and \href{https://www.math.univ-toulouse.fr/~gcebron/STARS.php}{STARS}, grant number ANR-20-CE40-0008, and by the PHC program \emph{Star} (Applications of random matrix theory and abstract harmonic analysis to quantum information theory).
    A.S.~was supported by the ANR project  \href{https://qtraj.math.cnrs.fr/}{Quantum Trajectories}, grant number ANR-20-CE40-0024-01.
    Z.O.~would like to extend my sincere gratitude to Professor Ion Nechita and Anna Szczepanek for their invaluable guidance and insightful discussions. Additionally, I also wish to express my appreciation to all professors at IMT for their exceptional guidance and support, which makes my M1 year in Toulouse a truly fulfilling and enriching experience.

    \bibliography{unistoch}
    \bibliographystyle{alpha}
    \bigskip
    \hrule
    \bigskip

\end{document}